\documentclass[12pt]{amsart}
\usepackage{cite}
\usepackage{amsmath}
\usepackage{amsthm}
\usepackage{amssymb}
\usepackage{amsfonts}
\usepackage{array}
\usepackage{fullpage}
\usepackage{tikz-cd}
\usepackage{enumerate}
\usepackage{multicol}
\usepackage{geometry}
\usepackage{verbatim}
\usetikzlibrary{shapes.geometric, arrows}
\theoremstyle{plain}
\newtheorem*{theorem*}{Theorem}
\newtheorem*{principle*}{Principle}
\newtheorem{theorem}{Theorem}
\newtheorem{lemma}[theorem]{Lemma}
\newtheorem{proposition}[theorem]{Proposition}
\newtheorem{corollary}[theorem]{Corollary}

\newtheorem*{lemma*}{Lemma}
\newtheorem*{proposition*}{Proposition}
\newtheorem*{corollary*}{Corollary}
\newtheorem*{conjecture*}{Conjecture}
\newtheorem*{definition*}{Definition}
\theoremstyle{remark}
\newtheorem{remark}[theorem]{Remark}

\theoremstyle{definition}
\newtheorem{example}{Example}[section]

\tikzstyle{startstop} = [rectangle, rounded corners, minimum width=3cm, minimum height=1cm, text centered, text width= 4cm, draw=black]
\tikzstyle{io} = [rectangle, minimum width=3cm, minimum height=1cm, text centered, text width=3cm, draw=black]
\tikzstyle{process} = [rectangle, minimum width=3cm, minimum height=1cm, text centered, text width=4cm, draw=black]
\tikzstyle{decision} = [rectangle, minimum width=3cm, minimum height=1cm, text centered, text width=4cm, draw=black]
\tikzstyle{arrow} = [thick,->,>=stealth]

\author{David Wen}
\title{Automorphisms of Minimal Surfaces of General Type with $K_S^2 = 1$, $p_g = 2$}
  \address{David Wen, National Center for Theoretical Sciences, No. 1 Sec. 4 Roosevelt Rd., National Taiwan University , Taipei, 106, Taiwan}
  \email{dwen@ncts.ntu.edu.tw}
 
\begin{document}

\newcommand{\bigslant}[2]{{\raisebox{.2em}{$#1$}\left/\raisebox{-.2em}{$#2$}\right.}}

\thanks{}

\begin{abstract} 
We classify the automorphism group of minimal surfaces of general type with $K_S^2 = 1$ and $\rho_g = 2$. Furthermore, we show that the order of the automorphism group is bounded above by 200 and can only have prime factors $p \leq 31$ with $p \neq 29$.
\end{abstract}

\maketitle

\section{Introduction}

Surfaces of general type are classically studied objects and much of the development in the understanding of them are due, in part, to the idea that general type surfaces are the surface analog of curves of genus $g \geq 2$. This can be seen in the many theorems of canonically polarized curves that has a surface analog for general type surfaces. One such theorem, by Xiao, generalizes the classical theorem of Hurwitz which bounds the automorphism group of a curve, $C$, with $g(C) \geq 2$ by $|Aut(C)| \leq 42(2g(C) - 2)$.
\begin{theorem*}[{\cite[Thm. 1]{Xiao94}} ]
Let $S$ be a minimal surface of general type then:
\[
|Aut(S)| \leq 42^2 \cdot K_S^2
\]
\end{theorem*}
Very few surfaces actually realize this bound. In fact, the only ones that do arise as quotients of products of Hurwitz curves, which are genus $g \geq 2$ curves that realizes the bounds of automorphism group of curves.

Another such theorem, of curve to surface generalization, by Bombieri highlights the exceptional cases of surfaces of general type based $m$-canonical map as follows:
\begin{theorem*}[{\cite{Bombieri73}, \cite[Thm. 1.1 (ii)]{Catanese87}} ]
Let $X$ be the canonical model of a surface of general type and $\psi_m$ the m-th canonical map, then $\phi_m$ is birational for $m\geq 3$ with the exception of $(K_X^2, p_g) = (1,2)$ and $(2,3)$.
\end{theorem*}
This implies that surfaces of type $(K_X^2, p_g) = (1,2)$ and $(2,3)$ behave like curves of genus $2$, at least as the $m$-canonical map is concerned, since the bicanonical map of a genus 2 curve is not birational. For the case of the exceptional surfaces above, the reason for the $3$rd-canonical map not being birational is the same as the case of curves of genus $2$ and that is due to there not being enough sections to give a birational map. As a result, these $m$-canonical maps actually give a ''nice" description of the canonical models of type $(K_X^2, p_g) = (1,2)$ and $(2,3)$ similar to how curves of genus $2$ are hyperelliptic. 

The focus of this paper is on minimal surfaces of general type, $S$, with $K_S^2 = 1$ and $p_g(S) = 2$. These surfaces are interesting due to being one of the exceptional case of Bombieri's Theorem above, which lends itself to an explicit description. Furthermore, we find that these surfaces can appear in fiber structures of threefolds of general type, thus lending itself to studying higher dimensional general type varieties that admit such fibrations.

In this paper, we will describe and classify the automorphism group of minimal surfaces of general type with $K_S^2 = 1$ and $p_g = 2$. In particular, we will prove the following:
\begin{theorem}[Theorem \ref{main}]
Let $S$ be a minimal surface of general type such that $K_S^2 = 1$ and $p_g = 2$, then we have the following:
\begin{enumerate}
	\item $|Aut(S)| \leq 200$
	\item $Aut(S)$ is isomorphic to one of the following:
		\[
		C_m \hspace{1cm} C_m \times C_n \hspace{1cm} C_m \rtimes T
		\]
		with $2 \leq m \leq 10$ and $2 \leq n $, where $C_t$ is the cyclic group of order $t$ and $T$ is isomorphic to one of the following:
		\begin{itemize}
			\item $C_k$ with $k \leq 10$
			\item $D_k$, the dihedral group, with $k \leq 10$ or with order at most $20$
			\item $S_4$
			\item $A_5$
			\item If $m \leq 8$ then possibly $S_5$
		\end{itemize}
	\item If $p$ is a prime number that divides $|Aut(S)|$, then $p \leq 31$ and $p \neq 29$.
\end{enumerate}
\end{theorem}

The proof of the main theorem will follow a sequence of reduction steps to make the automorphism group of $S$ more understandable and computable. Section 2 will begin the discussion on some basic properties of $Aut(S)$ and will show that $Aut(S)$ injects into $GL(2,\mathbb{C})$. Section 3 handles the extreme cases of these surface and show that it is in these cases where $|Aut(S)|$ can have large prime factors. Section 4 considers the general case by translating problem of computing $Aut(S)$ into a problem of computing Mobius transforms that fixes configuration of points on the Riemann sphere. Section 5 concludes the paper by combining the work of the previous sections and also gives further explanations.

\section{Preliminary Results}

\subsection{Minimal Surfaces to Canonical Surfaces}
If $S$ is a minimal surface of general type then $Aut(S) \cong Aut(S_{can})$, thus we can reduces the study of the automorphism group of $S$ to the automorphism group of it's canonical model. For the case of $K_S^2 = 1$ and $p_g(S) = 2$ we have the following description of $S_{can}$.

\begin{proposition}[{\cite[Example 1.3]{Catanese87}} ]
Let $S$ be a minimal surface of general type with $K_S^2 = 1$ and $p_g = 2$, then $S_{can}$ is a hypersurface in $\mathbb{P}(1,1,2,5)$ defined by the degree 10 weighted homogeneous polynomial  $w^2 - F_{10}(x,y,z)$ with at worst canonical singularities.
\end{proposition}

This sets up our convention for the paper. Let $\mathbb{P} := \mathbb{P}(1,1,2,5)$ and $S$ be a surface with canonical singularities defined by:
\[
S := Z(w^2 - F_{10}(x,y,z)) \subset \mathbb{P}
\]
where $x,y$ is of degree 1, $z$ is degree 2, $w$ is degree 5 and $F_{10}$ a degree 10 weighted homogeneous polynomial. This reduction let's us realize $S$ as a hypersurface in $\mathbb{P}$ of high degree which we can use to encode $Aut(S)$ into $Aut(\mathbb{P})$ with the following:

\begin{proposition}
Automorphisms of $S$ extend to automorphisms in $\mathbb{P}$.
\end{proposition}

\begin{proof}
As $S$ is a degree $10$ hypersurface in $\mathbb{P}$, then for all $m \leq 9$
\[
H^0(\mathbb{P}, \mathcal{O}(m)) \xrightarrow{\sim} H^0(S, \mathcal{O}(m))
\]
and $H^0(S, \omega_S^{\otimes m}) \cong H^0(S,\mathcal{O}(m))$. Furthermore $Aut(S)$ acts on $H^0(S, \omega_S^{\otimes m}) \cong H^0(\mathbb{P}, \mathcal{O}(m))$ and $H^0(\mathbb{P}, \mathcal{O}(5))$ is very ample, so $Aut(S)$ extends to an automorphism of $\mathbb{P}$.
\end{proof}
Thus automorphisms of $S$ extend to automorphisms of $\mathbb{P}$ that maps $S$ to itself. Equivalently, this is the same as asking for an automorphism of the graded ring $\mathbb{C}[x,y,z,w]$, with appropriate grading, that is semi-invariant on $w^2 - F_{10}(x,y,z)$. Automorphisms of $\mathbb{P}$ can be completely understood from it's action on $x,y,z$ and $w$ and must preserve degrees. Throughout this paper, we will treat automorphisms of $S$ and $\mathbb{P}$ as automorphisms of corresponding graded ring associated to them.

\subsection{Automorphisms of $S$ via automorphisms of $\mathbb{P}$}

To understand the automorphisms of $S$, we will utilize our understanding of automorphisms of $\mathbb{P}$. Automorphisms of $\mathbb{P}$ arise from the automorphisms of the graded ring $\mathbb{C}[x,y,z,w]$, under the grading of $(1,1,2,5)$ for $(x,y,z,w)$ modulo the $\mathbb{C}^*$-action. If we write
\[
\mathbb{C}[x,y,z,w] = \oplus_{i = 0}^\infty V_i
\]
where $V_i$ is the $i$-th graded part, then an automorphism $\phi \in Aut(\mathbb{P})$ can be realized as a sequence of linear maps $(A_i: V_i \rightarrow V_i)_{i}$ that must satisfy polynomial equations that respect the linear maps of lower degree grading. For example, $V_1$ is generated by $\{ x,y \}$ while $V_2$ is generated by $\{ x^2, xy, y^2, z \}$, so $A_2: V_2 \rightarrow V_2$ when restricted to $\{ x^2, xy, y^2 \}$ must agree with the action defined by $A_1: V_1 \rightarrow V_1$ when extended to the subring $\mathbb{C}[x,y]$. As $\mathbb{C}[x,y,z,w]$ is generated with degree $1,2$ and $5$ components, $\phi$ can be completely understood via:
\[
(A_1, A_2, A_5) \in GL(V_1) \times GL(V_2) \times GL(V_5)
\]
Since the equations between the variables of $A_1$, $A_2$ and $A_5$ are polynomial, there is some algebraic subgroup $H \leq GL(V_1) \times GL(V_2) \times GL(V_5)$ that satisfies these conditions in defining an automorphism of $\mathbb{P}$. Lastly, to obtain $Aut(\mathbb{P})$, we quotient out the $\mathbb{C}^*$-action on $H$. This gives the following diagram set up for this section:
\[
\begin{tikzcd}
\displaystyle
& H \arrow[d] \arrow[r, hook] & GL(V_1) \times GL(V_2) \times GL(V_5)\\
Aut(S) \arrow[r, hook] & Aut(\mathbb{P}) \cong H \slash \mathbb{C}^* & \\
\end{tikzcd}
\]

In this section, to get a better understanding of $Aut(S)$, we will construct a group morphism $Aut(S) \hookrightarrow H$ which would make the diagram commute. In particular, we will complete the diagram as follows:
\[
\begin{tikzcd}
\displaystyle
GL(2,\mathbb{C}) \arrow[r, hook] & H \arrow[d] \arrow[r, hook] & GL(V_1) \times GL(V_2) \times GL(V_5)\\
Aut(S) \arrow[r, hook] \arrow[u, hook] & Aut(\mathbb{P}) \cong H \slash \mathbb{C}^* & \\
\end{tikzcd}
\]

To accomplish this, we first need the following proposition on finite subgroups of $Aut(\mathbb{P})$.
\begin{proposition}
\label{finAutWProj}
Let $G \leq Aut(\mathbb{P})$ be a finite subgroup, then for all $\psi \in G$, there for each $\phi \in H$ such that $\phi \mapsto \psi$ the following holds:
\begin{align*}
\phi(w) &= \alpha w\\
\phi(z) &= \beta z\\
\phi(x) &= ax +cy\\
\phi(y) &= bx + dy
\end{align*}
for some $\alpha, \beta \in \mathbb{C}^*$ and
\[
A = 
\begin{pmatrix}
a & b \\
c & d
\end{pmatrix}\in GL_2(\mathbb{C})
\]
such that $A$ is diagonalizable.
\end{proposition}

\begin{proof}
Given $\psi \in G$ and $\phi \in H$ as above, we let $k$ be the order of $\psi$. By definition, $\phi$ must act on $\mathbb{C}[x,y]_1$, by some $A_1 \in GL(\mathbb{C}[x,y]_1)$ on the vectorspace of homogeneous linear polynomials, such that $A^k = t \cdot I_2$ with $t \in \mathbb{C}^*$. Thus $A$ is diagonalizable.

As $A$ is diagonalizable, there exists some $Q \in GL(\mathbb{C}[x,y]_1)$ such that $Q^{-1}AQ = D$ where $D$ is a diagonal matrix. Then up to an automorphism of $\mathbb{P}(1,1,2,5)$ using the change of basis matrix $Q$ on $x$ and $y$, we can reduce to the case where $\phi(x) = ax$ and $\phi(y) = dy$ with $a^k = d^k = t$. Now the homogenous degree 2 polynomials in the graded ring, $\phi$ has the following action on the generators:
\begin{align*}
\phi(z) &= \beta z + \gamma x^2 + \delta y^2 + \eta xy\\
\phi(x^2) &= a^2x^2\\
\phi(xy) &= adxy\\
\phi(y^2) &= d^2y^2
\end{align*}
As $\psi^k = id_{\mathbb{P}}$ this implies $\phi^k(z) =  t^2 z$ which sets up the following equation:
\[
\beta^kz + \sum_{i = 0}^k \gamma \beta^{k - i}a^{2i}x^2 +  \sum_{i = 0}^k \delta \beta^{k - i}d^{2i}y^2 +  \sum_{i = 0}^k \eta \beta^{k - i}(ad)^i xy = t^2 \cdot z 
\]
with $\beta^k = a^{2k} = (ad)^k = d^{2k} = t^2$. As $\{z,x^2,xy,y^2\}$ is a basis for the degree 2 component of the graded ring, this gives the following system of equations:
\begin{align*}
\beta^k &= t^2\\
\sum_{i = 1}^{k} \gamma \beta^{k - i}a^{2i} &= 0\\
\sum_{i = 1}^{k} \delta \beta^{k - i}d^{2i} &= 0\\
\sum_{i = 1}^{k} \eta \beta^{k - i}(ad)^{i} &= 0
\end{align*}
From the second equation, we can factor out $\gamma$:
\[
\gamma \left ( \sum_{i = 1}^{k} \beta^{k - i}a^{2k} \right ) = 0
\]
and assume $\displaystyle \sum_{i = 1}^{k} \beta^{k - i}a^{2k} = 0$ then:
\[
(\beta - a^{2}) \left ( \sum_{i = 1}^{k} \beta^{k - i}a^{2k} \right ) = \beta^{k+1} - a^{2k+ 2} = 0
\]
From before, $a^{2k} = \beta^k = t^2$, which implies $t^2(\beta - a^2) = 0$. As $t^2 \neq 0$, then $\beta = a^2$ but:
\[
0 = \gamma \sum_{i = 1}^{k} \beta^{k - i}a^{2i} = \gamma \sum_{i = 1}^{k} \gamma \beta^k = \gamma \cdot k \cdot t^2
\]
which is only possible if $\gamma = 0$. The same argument will show that $\delta = \eta = 0$. Thus $\phi(z) = \beta z$ with $\beta \in \mathbb{C}^*$. We can then apply the same argument to the degree 5 component of the graded ring to show that $\alpha^k = t^5$ and that $\phi(w) = \alpha w$. 
\end{proof}

As $Aut(S)$ is finite, the elements of the pullback of automorphisms of $S$, treated as automorphisms of $\mathbb{P}$, can be realized in the above form. To show that $Aut(S)$ can  be embedded into $GL(2,\mathbb{C})$, we need to use the restriction that $S$ only has canonical singularities.

\begin{proposition}[\cite{Reid87} ]
Let $S \subset \mathbb{P}$ is a canonical surface defined by $w^2 - F_{10}(x,y,z)$, then $F_{10}$ must have the term $a z^5$ with $a \neq 0$.
\end{proposition}

\begin{proof}
Assume otherwise then $S \subset \mathbb{P}(1,1,2,5)$ contains $P := (0:0:1:0) \in \mathbb{P}$, which corresponds to a quotient singularity of type $\frac{1}{2}(1,1,1)$ in the $z$-chart of $\mathbb{P}(1,1,2,5)$ and we will proceed to show that $P \in S$ is a singularity that is worst than canonical.

Following the notation in \cite[Sec. 4]{Reid87}, localizing with respect to $z$ gives a chart in $\mathbb{P}$ with the quotient singularity $\frac{1}{2}(1,1,1)$ and $S$ is locally the hypersurface in $\mathbb{P}$ is defined by the equation of the form:
\[
w^2 - f(x,y)
\]
where the degree of the monomials of $f$ is bounded between $2$ and $10$. Computing the discrepancy of the toric resolution of the total space as in \cite[Thm 4.6]{Reid87}, produces an exceptional divisor with discrepancy:
\[
a = \alpha(wxy) - \alpha(w^2 - f(x,y)) - 1
\]
where $\alpha(g)$ is the minimal weight associated to lattice of monomials that appear in $g$ that is contained in the positive cone associated to the affine toric variety. Computing this, we get $\alpha(wxy) = \frac{3}{2}$ and $\alpha(w^2 - f(x,y)) \geq 1$. In which case, $a < 0$  which means that $P \in S$ is not a canonical singularity.
\end{proof}

\begin{proposition}
\label{form}
With the above notation, if $G = Aut(S)$ then for each $\psi \in G$ we can find a unique $\phi \in H$  such that $\phi \mapsto \psi$ where $\phi$ is defind by the following action:
\begin{align*}
\phi(w) &= w\\
\phi(z) &= z\\
\phi(x) &= ax +cy\\
\phi(y) &= bx + dy
\end{align*}
with
\[
A = 
\begin{pmatrix}
a & b \\
c & d
\end{pmatrix}\in GL(2, \mathbb{C})
\]
such that $A$ is diagonalizable with eigenvalues being roots of unity.
\end{proposition}

\begin{proof}
We will use the fact that the $\mathbb{C}^*$-action is the kernal of the surjective morphism $H \rightarrow Aut(\mathbb{P})$, and so applying the $\mathbb{C}^*$-action on a representative will produce the same element of $Aut(\mathbb{P})$.

Via proposition \ref{finAutWProj}, automorphisms of finite order of $\mathbb{P}(1,1,2,5)$ has representatives in $H$, defined by the following action on the generators:
\begin{align*}
\phi(w) &= \alpha w\\
\phi(z) &= \beta z\\
\phi(x) &= ax +cy\\
\phi(y) &= bx + dy
\end{align*}
where $\alpha, \beta \in \mathbb{C}^*$ and
\[
A = 
\begin{pmatrix}
a & b \\
c & d
\end{pmatrix}\in GL_2(\mathbb{C})
\]
such that $A$ is diagonalizable. 

Let $\psi \in Aut(S) \leq Aut(\mathbb{P}(1,1,2,5))$ then we can choose $\phi \in H$ with the form above such that $\phi \mapsto \psi \in Aut(\mathbb{P})$. Then $\phi$ must leave $w^2 - F_{10}(x,y,z)$ semi-invariant so:
\[
\phi(w^2 - F_{10}(x,y,z)) = \gamma (w^2 - F_{10}(x,y,z))
\]
for some $\gamma \in \mathbb{C}^*$. Let $\gamma_0 := \zeta_{10}\gamma$ where $\zeta_{10}$ is a primitve $10$-th root of unite, then we can choose another representative of $\psi$ by applying the $\mathbb{C}^*$-action:
\[
\gamma_0: (x,y,z,w) \mapsto (\gamma_0x, \gamma_0y, \gamma_0^2 z, \gamma_0^5 w)
\]
So that $\gamma_0 \circ \phi$ is, in fact, invariant on $w^2 - F_{10}(x,y,z)$. So without loss of generalities, we can assume that $\phi$ is invariant on the polynomial $w^2 - F_{10}(x,y,z)$.

Since $w^2$ and $z^5$ appears in the equation that defines $S = Z(w^2 - F_{10}(x,y,z))$, this implies that the action of $\phi$ on $w$ and $z$ must be of the following form:
\begin{align*}
w &\mapsto \pm w\\
z &\mapsto  \zeta_5^k z
\end{align*}
where $\zeta_5$ is a fifth root of unity. From here, again choose another representative in $H$ using the $\mathbb{C}^*$-action by specifically using a power of a tenth root of unity we can choose $\phi$ to act on $w$ and $z$ as:
\begin{align*}
w &\mapsto w\\
z &\mapsto z
\end{align*}
as the $\mathbb{C}^*$-action only scales $x,y$, we will have that $A$ is still a diagonalizable matrix. To show it has eigenvalue that is a root of unity notice that since $\phi$ fixes $w^2 - F_{10}(x,y,z)$, then it must be true that $A^k = 1$. In which case, this is only possible if the eigenvalues are roots of unity.

Lastly, uniqueness of $\phi$ comes from the fact that if there was another representative $\phi_0$ of the same form that maps to $\psi$, then $\phi \circ \phi_0^{-1}$ is the identity and both maps have to fix $w^2 - F_{10}$. this results in the associated matrix $A, A_0 \in GL(2,\mathbb{C})$ satisfying $AA_0^{-1} = I_2$, so that $A = A_0$ and thus $\phi = \phi_0$.
\end{proof}

This shows that $Aut(S)$ can be realized solely by it's action on $x$ and $y$ and embeds as a finite subgroup of $GL(2,\mathbb{C})$ acting on $\mathbb{C}[x,y]$. Without loss of generalities, we will assume that $Aut(S)$ has the form above with $A \in GL(2,\mathbb{C})$ with $A^k = I_2$.

\subsection{Center of $Aut(S)$}
Having the above realization of $Aut(S) \hookrightarrow GL(2,\mathbb{C})$, allows us to use linear algebra in studying $Aut(S)$. In particular, we have a good understanding of some of the elements of the center of $Aut(S)$ inherited from $GL(2,\mathbb{C})$. In the following, let $\phi \in Aut(S)$ be an automorphism of $S$ with associated $A \in GL(2,\mathbb{C})$.

\begin{lemma}
\label{scale}
If the eigenvalues of $A$ given above are both equal to the same root of unity $\lambda$, then $\lambda^{2k} = 1$, for some $k \in \{0,1,2,3,4,5\}$.
\end{lemma}

\begin{proof}
If the eigenvalue of $A$ are the same then $A = \lambda \dot  I_2$, which is a scaling matrix by $\lambda$. As $A$ is associated with $\phi \in Aut(S)$, then proposition \ref{form} implies that it acts $x,y$ which has a grading by $z^i$. We can write:
\[
F_{10}(x,y,z) = \sum_{i = 0}^5 q_i(x,y)z^i
\]
with $q_i(x,y)$ being a homogenous polynomial of $\mathbb{C}[x,y]$ of degree $10-2i$. Thus the only way to fix $w^2 - F_{10}(x,y,z)$, is to fix $q_i(x,y)$, in which case we can only have this if $\lambda^{2k} = 1$ where $k = 5 - i$, for every non-constant $q_i(x,y)$ that appears in $F_{10}(x,y,z)$.
\end{proof}

\begin{remark}
The above lemma puts a restriction on possible scaling matrix associated to $Aut(S)$ based on the which $q_i(x,y) \neq 0$. For example, if $q_0$ and $q_1$ are chosen generically, then $\lambda = \pm 1$. Otherwise, if only $q_1$ is non-zero then $\lambda = \zeta_8^k$ where $\zeta_8$ is a primitive $8$-th root of unity.
\end{remark}

\begin{corollary}
Let $\phi \in Aut(S)$ with the associated $A$ as above, then $\phi^{2k} = id_S$, for some $k \in \{0,1,2,3,4,5\}$. Furthermore, $\phi$ is in the center of $Aut(S)$.
\end{corollary}

\begin{corollary}
\label{center}
Let $G \leq Aut(S)$ be the subgroup generated by associated scaling matrices in $GL(2,\mathbb{C})$, then $|G| \leq 10$.
\end{corollary}

This shows there is only a finite number of possible scaling factors allowed for automorphisms of $S$ realized in the form as in proposition \ref{form}. In particular, other than these scaling matrices we only need to understand the action of $Aut(S)$ on $q_i(x,y)$ to completely understand the automorphism group of $S$. From the fundamental theorem of algebra, $q_i$, being homogeneous of two variables, factors completely into linear factors. So $Aut(S)$ can be understood via it's actions on the linear homogeneous polynomials. From this observation, we have the following lemma.

\begin{lemma}
Let  $Y = \{l_k\}$ be the linear factors for $q_i$ for $0 \leq i \leq 4$, then the action of $\phi \in Aut(S)$ on $Y$ can only fix at most two distinct elements of $Y$ or must fix all of them.
\end{lemma}

\begin{proof}
The action of $\phi$ on $Y$ is induced from the action of $A$ on $\mathbb{C}[x,y]_1$ the degree 1 grading of the polynomial ring. For $A$ to fix a linear factor it is equivalent to scaling the linear factor which means that $l_j$ is a eigenvector of $A$. As $A$ is invertible and diagonalizable, it can have either two distinct one dimensional eigenspaces or a two dimensional eigenspace which corresponds respectively to fixing two distinct linear factors or fixing all of them.
\end{proof}

The elements of $Y$ are mapped back into $Y$ under the action of $Aut(S)$. From this, there are two cases to consider, the first is where $|Y| = 2$ where the automorphism will be primarily scaling each linear factors. The second is when $|Y| \geq 3$, in which case, $Aut(S)$ permute the elements of $Y$ which can be understood by identifying as Mobius transformation acting on points on a Riemann Sphere. Note that we can not have $|Y| = 1$ since implies the surface $S$ is not normal.

\section{Two Factor Situation}                                                 
In this section, we consider the case of $|Y| = 2$ and show that Theorem \ref{main} holds. Before we proceed, we will need the following:

\begin{lemma}
Let $S$ be a canonical surface as previous defined with the defining equation:
\[
w^2 - \sum_{i = 0}^5 q_i(x,y)z^i
\]
in $\mathbb{P}$, then $q_0(x,y) \neq 0$ or $q_1(x,y) \neq 0$.
\end{lemma}

\begin{proof}
Assume otherwise with $q_0(x,y) = q_1(x,y)= 0$, then the singular locus of $S$ contains a $1$-dimensional component, which implies that $S$ is not a normal surface which is a contradiction.
\end{proof}

\begin{proposition}
\label{twoFactors}
Assume that $|Y| = 2$, then $|Aut(S)| \leq 200$.
\end{proposition}

\begin{proof}
Without loss of generalities, we can assume that the linear factors are $x$ and $y$ by a change of basis that maps the two linear factors to $x$ and $y$. So the equation is of the form:
\[
w^2 - \sum_{i = 0}^5 a_i x^{n_i}y^{m_i}z^i
\]
where $n_i + m_i = 10-2i$. In this form, the following is true:
\begin{itemize}
	\item $a_5 \neq 0$ otherwise the surface has worst than canonical singularities.
	\item $a_0 \neq 0$ or $a_1 \neq 0$ otherwise the resulting surface is not normal.
	\item At least $3$ of the $a_i$ are non-zero otherwise there is an infinite number of automorphisms.
	\item We can not have $n_i = m_i$ for all $a_i \neq 0$ as this would also allow for an infinite number of automorphisms. This shows that the automorphisms can not swap the linear factors that are being acted upon.
	\item Automorphisms only act by scaling $x$ and $y$, thus they are eigenvectors of the matrix $A$ associated to automorphism $\phi_A \in Aut(S)$.
\end{itemize}
Under these conditions, a bound can be computed for the order of the automorphism group for such surfaces. The calculations break down by the degree of $z$ and, since $z^5$ must be one of the terms that is non-zero, we only need to consider the following pairs of degrees $(p,q)$ where $p = 0,1$ and $p \neq q = 1,2,3,4$ for the other $z$ graded monomials. If there are more than three pairs then these calculations would provide an upper bound as more monomials would introduce more restrictions on these equations. 

So let $(p,q)$ be as above, then there is an $n_p,m_p,n_q,m_q \in \mathbb{Z}_{\geq 0}$ such that $n_p + m_p = 10-2p$ and $n_q + m_q = 10-2q$ where,
\[
w^2 - a_5z^5 - a_px^{n_p}y^{m_p}z^p - a_qx^{n_q}y^{m_q}z^q
\]
is a part of the equation $w^2 - F_{10}(x,y,z)$, with $a_5, a_p, a_q \neq 0$. Now any automorphism $\phi \in Aut(S)$, acts as follows $\phi(x) = \alpha x$ and $\phi(y) = \beta y$ and must fix the above polynomial. This sets up the following equations:
\begin{align*}
\alpha^{n_p}\beta^{m_p} &= 1\\
\alpha^{n_q}\beta^{m_q} &= 1
\end{align*}
There are only finitely many possibilities for $n_p,m_p,n_q$ and $m_q$ with the following constraints:
\begin{itemize}
	\item $n_p,m_p,n_q,m_q \in \mathbb{Z}_{\geq 0}$
	\item $n_p + m_p = 10-2p$
	\item $n_q + m_q = 10-2q$
\end{itemize}
Solving for $\beta$ gives the following equality:
\[
\beta^{n_qm_p - n_pm_q} = 1
\]
We can apply calculus to compute the maximum values of $N := |n_qm_p - n_pm_q|$. In general, we will be finding the absolute maximum or minimum of $f(x,y,z,w) = xy-zw$ bounded in the region defined by $x,y,z,w \geq 0$ and $x + z = u$ and $y + w = v$. with $u, v \in \mathbb{N}$. Clearly, the absolute max and min can only be realized along the boundary of the region. So set $x = u - z$ and $y = v - w$ and simplify the problem to $f(z,w) = uv - u w - v z$. Since $w,z \geq 0$, the maximum is realized when $w = z = 0$. This results in $x = u$ and $y = v$. Now $p, q \in \{0,1,2,3,4,5\}$ with $p \neq q$, so the maximum for $u$ and $v$ would be $u = 10$ and $v = 8$. Thus $N \leq 80$. 

By symmetry, $\alpha^N = \beta^N = 1$ and by computing the minimum of the greatest common divisors of the possible pairs of $(n_p, n_q)$ and $(m_p, m_q)$, we get that the number of possible solutions for $\beta$ given a solution of $\alpha$ is at most $2$ with the exception of $(6,0)$ and $(4,8)$ but a quick verification will show that the number of solutions of:
\begin{align*}
\alpha^{8} &= 1\\
\alpha^{4}\beta^{6} &= 1
\end{align*}
 is $48$. Thus, the number of solutions is bounded above by $160$ which satisfies $|Aut(S)| \leq 200$.
\end{proof}

\begin{remark}
\label{form1}
From the computations above, the automorphism groups of these surfaces is isomorphic to one of the following:
\[
C_m \hspace{1cm} C_m \times C_n
\]
where $C_k$ is the cyclic group of $k$ elements and $m \leq 80$ and $n \leq 2$.
\end{remark}

\begin{example}
Consider the following surface $S \subset \mathbb{P}(1,1,2,5)$ defined by the weighted homogeneous polynomial
\[
w^2 - z^5 - xy^7z - x^9y
\]
This is a surface of general type with at worst canonical singularities. The automorphisms of such a surface can be realized by scaling $x$ and $y$ giving the following:
\begin{align*}
x &\mapsto \alpha x\\
y &\mapsto \beta y
\end{align*}
and this action must fix the defining equation, which gives the following system of equations:
\begin{align*}
\alpha \cdot \beta^7 &= 1 \\
\alpha^9 \cdot \beta &= 1 
\end{align*}
which gives $N = 7 \cdot 9 - 1 \cdot 1 = 62$. Thus $\alpha^{62} = \beta^{62} = 1$. So $\alpha$ and $\beta$ must be powers of the $62$nd-roots of unity. Furthermore, the minimal greatest common divisor of the exponents of $\beta$ is $1$ and the same for $\alpha$. So we have that for each $\beta$ there is a unique $\alpha$ solution. This shows that the automorphism group is order $62$. 
\end{example}

\begin{example}
Consider the following surface $S \subset \mathbb{P}(1,1,2,5)$ defined by the weighted homogeneous polynomial
\[
w^2 - z^5 - x^2y^6z - x^8y^2
\]
and an automorphism of such a surface realized by scaling:
\begin{align*}
x &\mapsto \alpha x\\
y &\mapsto \beta y
\end{align*}
Then we have the following equations:
\begin{align*}
\alpha^2 \cdot \beta^6 &= 1 \\
\alpha^8 \cdot \beta^2 &= 1 
\end{align*}
which gives $N = 6 \cdot 4 - 2 = 22$. Thus $\alpha^{22} = \beta^{22} = 1$. So $\alpha$ and $\beta$ must be roots of unity with order $22$. Furthermore, the minimal greatest common divisor of the exponents of $\beta$ is $2$ and the same for $\alpha$. So for each $\beta$ there is possibly 2 values for $\alpha$ which would satisfy this equation. Thus the order of this automorphism group is $44$. 
\end{example}

\section{Point Configurations on $\mathbb{CP}^1$}

Assume that $|Y| \geq 3$, then we can reinterpret the problem into a question on the Riemann sphere by treating the distinct linear factors, which we can write as $\{l_i = a_ix + b_iy\}_{i \in I} \subset \mathbb{C}[x,y]_1$ as the following points $\{(a_i : b_i)\} \subset \mathbb{CP}^1$ on the Riemann sphere.

\subsection{Algebraic picture}
\label{algPic}
Treating $\mathbb{C}[x,y]_1$ as a 2-dimensional vector space we can projectivise this space to obtain the Riemann sphere, so that each $l_i = a_ib + b_i y \mapsto (a_i : b_i)$. Then $Y$ maps to a configuration of points in $\mathbb{CP}^1$. From the previous discussion, we can realize an automorphism of $S$ by some $A \in GL(2,\mathbb{C})$, whose action of $\mathbb{C}[x,y]_1$ extends to an action on $\mathbb{C}[x,y]$ and fixes $q_i$. 

Reinterpreting this situation into the Riemann sphere setting, $A$, as induced by $\phi \in Aut(S)$, induces a Mobius transform that will map the points of $\{ (a_i : b_i) \}$ back onto themselves which has added conditions of preserving the grading by $z^i$ and multiplicities of the linear factors for each $q_i$. Lastly, $|Aut(S)| < \infty$, which means $Aut(S)$ must maps onto a finite subgroup of group of Mobius transformations.

Finite subgroups of  Mobius transformations are conjugate to finite subgroups of $SU(2) \cong SO(3, \mathbb{R})$. Furthermore, finite subgroups of $SO(3, \mathbb{R})$ are classified to be cyclic, dihedral or the symmetry group of platonic solids, which give us the following:
\begin{proposition}
Assuming $|Y| \geq 3$ and $K$ the image of $Aut(S)$ in group of Mobius tranformations, then $K$ is isomorphic to a cyclic group, dihedral group or the symmetry of a platonic solid.
\end{proposition}

As the equation defining $S$ must have $q_0 \neq 0$ or $q_1 \neq 0$ and the fact that the automorphisms of $S$ must map $q_i$ back to itself. We are primarily concerned with the case where the number of points counted with multiplicities comes from $q_0$ or $q_1$.
\begin{lemma}
Assume that $q_0$ or $q_1$ have more than three factors, then subgroup of Mobius transform that map the configuration of points obtained by mapping $q_0, q_1$ to $\mathbb{CP}^1$, which has degree 10 and 8 respectively, is isomorphic to either a cyclic group of order at most $10$, the dihedral group of at most $10$ elements or the symmetry group of the tetrahedron or octahedron/cube.
\end{lemma}

\begin{proof}
We have $q_0$ (resp. $q_1$) is of degree 10 (resp. $8$) so it will map up to $10$ (resp. 8) points in $\mathbb{CP}^1$ counting multiplicities. Taking the stereographic projection onto the unit sphere in $\mathbb{R}^3$ we see the configuration of points would only allow for, up to a Mobius tranformation, an $n$-gon with at most $n = 10$ and the only platonic solids that can be formed with at most $10$ points is the tetrahedron, octahedron and cube.
\end{proof}

\begin{remark}
If $q_0$ and $q_1$ has only one or two factors then we can adjoin it with $q_j$ for $j > 1$ to obtain at least 3 linear factors to map into $\mathbb{CP}^1$ and the lemma would still hold. 
\end{remark}

\begin{corollary}
Let $K$ be the image of $Aut(S)$ in the group of Mobius transformations, then $|K| \leq 24$.
\end{corollary}

\begin{proposition}
\label{moreTwoFactors}
Given $S$ as before with $|Y| \geq 3$, then $|Aut(S)| \leq 240$.
\end{proposition}

\begin{proof}
Let $M := PSL(2,\mathbb{C})$ be the group of Mobius tranformations on $\mathbb{CP}^1$, then from the previous discussion there is a group morphism $\phi: Aut(S) \rightarrow M$. We see that $|\phi(Aut(S))| \leq 24$ and $ker(\phi)$ are the elements of $Aut(S)$ that fixes the linear factors of $p(x,y)$ which comes from scaling and thus $|ker(\phi)| \leq 10$. Putting these two information together we get:
\[
|Aut(S)| = |\phi(Aut(S)| |ker(\phi)| \leq 240
\]
\end{proof}

The morphism of $Aut(S) \rightarrow M$ will give the second part of the classification of isomorphism classes of $Aut(S)$, the first part being remark \label{form1}
\begin{corollary}
\label{form2}
Under the assumption that $|Y| \geq 3$, we have that:
\[
Aut(S) \cong C_m \rtimes T
\]
where $C_k$ is a cyclic subgroup of with $k \leq 10$ and $T$ is a finite subgroup of the Mobius transformations.
\end{corollary}

For the final step to get to $|Aut(S)| \leq 200$, we need to show that the octahedral group action on the Riemann sphere is not compatible with mapping $10$ points back to itself on $\mathbb{CP}^1$ which is the only part of the equation of $S$ that allows a scaling by a $10$-th root of unity. 

\begin{lemma}
\label{final}
Let $T$ be the group of Mobius transformations that map $10$ points, counting multiplicities with at least 3 distinct points, back to itself then $|T| \leq 24$.
\end{lemma}

\begin{proof}
Given $T$ is a finite subgroup of Mobius tranformations, $H$ is conjugate to a subgroup of $PSU(2,\mathbb{C})$, \cite{Shurman97}, thus, without loss of generalities, we can assume $T$ is a subgroup of $PSU(2,\mathbb{C}) \cong SO(3, \mathbb{R})$. The Mobius transformation and stereographic projection will map the $10$ points to $10$ points on the sphere $S^2 \subset \mathbb{R}^3$. We see in this case that we are looking at the possibly rotational symmetries of $10$ points on the sphere and their possible configurations. 

The cyclic groups of order up to $9$ and dihedral group of order $20$ can be realized from the symmetry of the regular $9$-gon with a fixed point and symmetry of a decagon, respectively. So we only need to rule out the octahedral group and the icosahedral group. For the icosahedral group we need $12$ vertices to form an icosahedron embedded in $S^2 \subset \mathbb{R}^3$. As we only have $10$ points, we do not have enough points to form an icosahedron on $S^2 \subset \mathbb{R}^3$, thus we do not have enough points to realize the icosahedral group.

The octahedral group is the symmetry of the octahedron or the cube, which only needs $6$ and $8$ vertices respectively on $S^2\subset \mathbb{R}^3$. There is $10$ points to place so there will be $4$ and $2$ extra points respectively to place on $S^2$. If we allow multiplicities, then we can only map vertices with the same multiplicities to each other. In which case, with $10$ points we can only realize a proper subgroup of the octahedral group as the automorphisms that map $10$ points with multiplicities to itself and so will have order $\leq 12$. If we do not have multiplicites then the orbit of a point not on a rotational axis under the action of the octahedral group will be $24$ points which is more than the $10$ marked points on the sphere.

So we see that with $10$ points counting multiplicities on $\mathbb{CP}^1$, we can not realize the octahedral and icosahedral groups in mapping of these 10 points back to itself, so $|T| \leq 20$ when $q_0 \neq 0$.
\end{proof}

\subsection{Geometric picture}
\label{geoPic}

Lastly, in addition to the algebraic pictures, we can realize these arguments geometrically though the canonical map $\phi_{|K_S|}:S \dashrightarrow \mathbb{P}^1$ and, corresponding to the algebraic arguments, there are also a geometric arguments. In particular, we can realized the map of the equation $w^2 - F_{10}$ onto the $\mathbb{CP}^1$ by the canonical map $\phi_{|K_S|}: S \dashrightarrow \mathbb{CP}^1$. 

In this setting, the fibers over the points that vanish at one of the factors of $q_i$ are special fibers and $Aut(S)$ is equivariant on the canonical map, so $Aut(S)$ induces a Mobius transform on $\mathbb{CP}^1$ which must map these special fibers to themselves. This we realize as the subgroup of Mobius transforms that map $m$ marked points graded by multiplciites and singular fiber type to themselves as mentioned in \ref{algPic}

Continuing with the analog between the algebraic and geometric, there is a group morphism $\psi: Aut(S) \rightarrow Mob(\mathbb{CP}^1)$, and we understand  how $\psi(Aut(S))$ acts on the special fibers of $\phi_{|K_S|}: S \dashrightarrow \mathbb{CP}^1$ over the points of $\mathbb{CP}^1$ where some $q_i$ vanishes. Furthermore, $H := ker(\psi) \leq Aut(S)$ fixes the fibers of $\phi_{|K_S|}: S \dashrightarrow \mathbb{CP}^1$ since $H \mapsto id \in Mob(\mathbb{CP}^1)$ and from corollary \ref{center}, $H$ is contained in the center of $Aut(S)$ and is in fact a cyclic subgroup of even order at most $10$.

From this geometric picture, we show one more constraint on $H$:
\begin{proposition}
$H$, as defined above, is a cyclic subgroup of order $2,4,8$ or $10$.
\end{proposition}

\begin{proof}
To show this we only need to show that $|H| \neq 6$. We start by resolving the basepoint of $\phi_{|K_S|}$ to obtain a morphism $\phi:\hat{S} \rightarrow \mathbb{CP}^1$ which is a relatively minimal genus $2$ fibration, \cite[Lemma 2.1]{CatanesePignatelli2006}. Now, $H$ fixes the fiber as well as the exceptional locus of $\hat{S} \rightarrow S$, which means that $H$ is a relative automorphism of the genus 2 fibration and so can be realized as the automorphism of a genus 2 curve over the function field, $\mathcal{K}$, of $\mathbb{CP}^1$. As $H$ must fix the exceptional curve, it must also fix a point on this genus 2 curve over $\mathcal{K}$. As this exceptional curve is fixed by the involution on the genus 2 fibration, it must a ramification point of the genus 2 curve over $\mathcal{K}$. 

So $H$ acts on a genus $2$ curve which fixes a ramification point, thus elements of $H$ can only permute the other $5$ ramification points and can not fix just $2$ points. Looking at the cycles of permutations, if $\sigma \in H$ is of order $3$ then it must act on the $6$ ramification points on $\mathbb{P}_K^1$ as a composition of $1$-cycle, $2$-cycle, and $3$-cycle or three $1$-cycle and a $3$-cycle. In either of those cases we have that this can not be since then we have that $\sigma$ or $\sigma^2$ fixes three elements of $\mathbb{P}_\mathcal{K}^1$ and so must be the identity and is not of order $3$. Thus $H$ cannot have an element of order $3$ which means it cannot be of order $6$.
\end{proof}

\section{Conclusion}
From proposition \ref{twoFactors}, \ref{moreTwoFactors} and lemma \ref{final}, we have shown the following:
\begin{proposition}
\label{bound}
Let $S$ be a minimal surface of general type such that $K_S^2 = 1$ and $p_g = 2$, then:
\[
|Aut(S)| \leq 200
\]
\end{proposition}

The above bound of $200$ is in fact sharp from the following example:
\begin{example}
Consider the surface defined by the polynomial
\[
w^2 - z^5 - x^{10} - y^{10}
\]
then $Aut(S)$ is generated by scaling of $x$ and $y$ by tenth powers individually and swapping the $x$ and $y$ terms. From those generators, there are $200$ automorphisms. The analysis above will in fact show that $Aut(S) \cong C_{10} \times D_{10}$ where $C_{10}$ is the cyclic group of $10$ elements and $D_{10}$ is the dihedral group of order $20$.
\end{example}

In terms of the order of the automorphism group of $S$, the above analysis puts an upper bounds to the number of possible prime factors.

\begin{proposition}
\label{prime}
If $p$ is a prime number that divides $|Aut(S)|$, then $p \leq 31$ and $p \neq 29$. Furthermore, when $p \geq 11$ then $S$ is defined by an equation:
\[
w^2 - a z^2 - \sum_{i = 1}^5 q_i(x,y)z^i
\]
such that there are two homogeneous linear polynomials $L_1, L_2 \in \mathbb{C}[x,y]$ with $L_1 \neq \alpha L_2$ which are the only prime factors of $q_i$ for all $i$.
\end{proposition}

\begin{proof}
Computing the values of $N$ in proposition \ref{twoFactors}, produces a finite list defined below:
\[
N \in \{(10-q)p - q(k-p) | k \in \{2,4,6,8\}, 0\leq q \leq 10, 0 \leq p \leq k\}
\]
and the prime factors up to $31$ are realized in this list except for $29$. In computing the orbit types of $q_i$ in proposition \ref{moreTwoFactors}, the computation of orbits and stabilizers are restricted to factors $\leq 10$. So the prime factors are bounded above by $31$ and, furthermore, if $p > 10$, then we have to be in the case of proposition \ref{twoFactors}.
\end{proof}

\begin{example}
This bound on prime factors is sharp since the surface defined by $Z(w^2 - z^5 - xy^7z - x^9y) \subset \mathbb{P}(1,1,2,5)$ has canonical singularities and it's automorphism group is a cyclic group of order $62$ generated by the action:
\begin{align*}
x &\mapsto \zeta_{62}\\
y &\mapsto  \zeta_{62}^{53}
\end{align*} 
Thus this surface has an automorphism of order $31$. Furthermore we can realize all prime number $11 \leq p \leq 31$ and $p \neq 29$ by the following surfaces defined by the following equations:
\begin{itemize}
	\item $S = Z(w^2 - z^5 - x^2y^6z - x^8y^2)$, $Aut(S) \cong C_2 \times C_{22}$, subgroup of order $11$.
	\item $S = Z(w^2 - z^5 - x^2y^6z - x^9y)$, $Aut(S) \cong C_{52}$, subgroup of order $13$.
	\item $S = Z(w^2 - z^5 - xy^5z^2 - x^7yz)$, $Aut(S) \cong C_{34}$, subgroup of order $17$.
	\item $S = Z(w^2 - z^5 - xy^5z^2 - x^8y^2)$, $Aut(S) \cong C_{38}$, subgroup of order $19$.
	\item $S = Z(w^2 - z^5 - xy^7z - x^7y^3)$, $Aut(S) \cong C_{46}$, subgroup of order $23$.
	\item $S = Z(w^2 - z^5 - xy^7z - x^9y)$, $Aut(S) \cong C_{62}$, subgroup of order $31$.
\end{itemize}
\end{example}

Thus combining all the analysis of proposition \ref{bound}, proposition \ref{prime} with remark \ref{form1} and corollary \ref{form2}, we have a complete understanding of the automorphism group of minimal surfaces of general type with $K_S^2 = 1$ and $p_g = 2$, thus proving the main theorem.

\begin{theorem}
\label{main}
Let $S$ be a minimal surface of general type such that $K_S^2 = 1$ and $p_g = 2$, then we have the following:
\begin{enumerate}
	\item $|Aut(S)| \leq 200$
	\item $Aut(S)$ is isomorphic to one of the following:
		\[
		C_m \hspace{1cm} C_m \times C_n \hspace{1cm} C_m \rtimes T
		\]
		with $2 \leq m \leq 10$ and $2 \leq n $, where $C_t$ is the cyclic group of order $t$ and $T$ is isomorphic to one of the following:
		\begin{itemize}
			\item $C_k$ with $k \leq 10$
			\item $D_k$, the dihedral group, with $k \leq 10$ or with order at most $20$
			\item $S_4$
			\item $A_5$
			\item If $m \leq 8$ then possibly $S_5$
		\end{itemize}
	\item If $p$ is a prime number that divides $|Aut(S)|$, then $p \leq 31$ and $p \neq 29$.
\end{enumerate}
\end{theorem}

\bibliographystyle{plain}
\bibliography{refs}{}

\begin{thebibliography}{1}

\bibitem{Bombieri73}
E.~Bombieri.
\newblock Canonical models of surfaces of general type.
\newblock {\em Inst. Hautes \'{E}tudes Sci. Publ. Math.}, (42):171--219, 1973.

\bibitem{Catanese87}
F.~Catanese.
\newblock Canonical rings and ``special'' surfaces of general type.
\newblock In {\em Algebraic geometry, {B}owdoin, 1985 ({B}runswick, {M}aine,
  1985)}, volume~46 of {\em Proc. Sympos. Pure Math.}, pages 175--194. Amer.
  Math. Soc., Providence, RI, 1987.

\bibitem{CatanesePignatelli2006}
Fabrizio Catanese and Roberto Pignatelli.
\newblock Fibrations of low genus. {I}.
\newblock {\em Ann. Sci. \'{E}cole Norm. Sup. (4)}, 39(6):1011--1049, 2006.

\bibitem{Reid87}
Miles Reid.
\newblock Young person's guide to canonical singularities.
\newblock In {\em Algebraic geometry, {B}owdoin, 1985 ({B}runswick, {M}aine,
  1985)}, volume~46 of {\em Proc. Sympos. Pure Math.}, pages 345--414. Amer.
  Math. Soc., Providence, RI, 1987.

\bibitem{Shurman97}
Jerry Shurman.
\newblock {\em Geometry of the quintic}.
\newblock A Wiley-Interscience Publication. John Wiley \& Sons, Inc., New York,
  1997.

\bibitem{Xiao94}
Gang Xiao.
\newblock Bound of automorphisms of surfaces of general type. {I}.
\newblock {\em Ann. of Math. (2)}, 139(1):51--77, 1994.

\end{thebibliography}

\end{document}